\documentclass[12pt,reqno]{amsart}
\usepackage[margin=1in]{geometry}
\usepackage{amssymb,amsfonts,amsmath,mathrsfs,bm,bbm,booktabs,enumerate,etoolbox,mathtools}
\usepackage[breaklinks=true,colorlinks=true,linkcolor=teal,citecolor=teal,filecolor=teal,urlcolor=black!50!blue]{hyperref}
\hypersetup{linktocpage}
\usepackage[hiresbb]{graphicx,xcolor}
\allowdisplaybreaks[4]

\newtheorem{theorem}{Theorem}[section]
\newtheorem{conjecture}[theorem]{Conjecture}
\newtheorem{lemma}[theorem]{Lemma}
\newtheorem{definition}[theorem]{Definition}

\theoremstyle{definition}
\newtheorem{example}[theorem]{Example}
\newtheorem{examples}[theorem]{Examples}
\newtheorem{remark}[theorem]{Remark}

\numberwithin{equation}{section}

\renewcommand{\Re}{\mathrm{Re}}
\renewcommand{\epsilon}{\varepsilon}

\patchcmd{\section}{\scshape}{\bfseries}{}{}
\makeatletter
\renewcommand{\@secnumfont}{\bfseries}
\makeatother
\makeatletter\newcommand{\tpmod}[1]{{\@displayfalse \pmod{#1}}}

\begin{document}

\title{Towards the Deep Riemann Hypothesis for $\mathrm{GL}_{n}$}

\author{Ikuya Kaneko}
\address{The Division of Physics, Mathematics and Astronomy, California Institute of Technology, 1200 E. California Blvd., Pasadena, CA 91125, USA}
\email{ikuyak@icloud.com}
\urladdr{\href{https://sites.google.com/view/ikuyakaneko/}{https://sites.google.com/view/ikuyakaneko/}}

\author{Shin-ya Koyama}
\address{Department of Biomedical Engineering, Toyo University, 2100 Kujirai, Kawagoe, Saitama, 350-8585, Japan}
\email{koyama@tmtv.ne.jp}
\urladdr{\href{http://www1.tmtv.ne.jp/~koyama/}{http://www1.tmtv.ne.jp/~koyama/}}

\author{Nobushige Kurokawa}
\address{Department of Mathematics, Tokyo Institute of Technology, Oh-okayama, Meguro-ku, Tokyo, 152-8551, Japan}
\email{kurokawa@math.titech.ac.jp}

\thanks{IK acknowledges the support of the Masason Foundation.}

\date{\today}

\subjclass[2020]{11M06, 11M26 (primary); 11M38 (secondary)}

\keywords{Deep Riemann hypothesis, grand Riemann hypothesis, Birch--Swinnerton-Dyer conjecture, Chebyshev bias}

\begin{abstract}
We explicate the deep Riemann hypothesis for the general linear group $\mathrm{GL}_{n}$~on the convergence of normalised Euler products of standard $L$-functions on the critical line.~It conditionally improves upon the error term in the prime number theorem beyond what the grand Riemann hypothesis predicts. Furthermore, we discuss the Chebyshev bias for Satake parameters on $\mathrm{GL}_{n}$ from the perspective of the deep Riemann hypothesis.
\end{abstract}

\maketitle
\tableofcontents

\section{Introduction}
Riemann at 33 years of age communicated his findings to the Monatsberichte der Berliner Akademie on 19 October 1859, and Kummer read the paper at the meeting of the academy on 3 November 1859. The specific date it appeared in print, which should be the official~date of the paper, is not known. The Riemann hypothesis (RH) was presented in his paper~\cite{Riemann1859}, and it remains one of the most difficult unsolved problems in mathematics. RH is a problem of the real part of nontrivial zeros of the Riemann zeta function. The third author suggested in his Japanese books~\cite{Kurokawa2012,Kurokawa2013} the deep Riemann hypothesis (DRH) pertaining to the convergence of Euler products on the critical line, which serves as a prediction exceeding RH. If RH is likened to the mountain climbing, then it falls upon the fifth station, and studies of zeta functions mean to climb Mount Zeta.

Riemann~\cite{Riemann1859} attempted to prove the prime number theorem that states as $x \to \infty$,
\begin{equation}\label{Prime-Number-Theorem}
\pi(x) \coloneqq \#\{p: p \leq x \} \sim \frac{x}{\log x},
\end{equation}
where $\pi(x)$ counts the number of primes up to $x$. This was proposed by Legendre and Gau{\ss} around the end of the 18th century. Riemann introduced complex analysis to delve into~$\pi(x)$. The full proof was provided half a century later independently by Hadamard~\cite{Hadamard1893,Hadamard1896} and de la Vall\'{e}e Poussin~\cite{delaValleePoussin1896}. We direct the reader to~\cite{Zagier1997} for an elementary proof.

The Riemann zeta function $\zeta(s)$ in the real variable $s$ was introduced by Euler in his tome \textit{Introductio in Analysin Infinitorum}. He studied the values of $\zeta(s)$ at even integers with the divination of a functional equation. Riemann~\cite{Riemann1859} obtained the functional equation with a deeper analysis of $\pi(x)$ by considering the zeta function in the complex variable. We define
\begin{equation*}
\zeta(s) \coloneqq \sum_{n = 1}^{\infty} \frac{1}{n^{s}} = \prod_{p} \left(1-\frac{1}{p^{s}} \right)^{-1}, \qquad \Re(s) > 1,
\end{equation*}
where the product runs through primes. He then established the functional equation
\begin{equation*}
\Lambda(s) \coloneqq \pi^{-\frac{s}{2}} \Gamma \left(\frac{s}{2} \right) \zeta(s) = \Lambda(1-s).
\end{equation*}
The paper of Riemann~\cite{Riemann1859} is a r\'{e}sum\'{e} of his extensive work on $\zeta(s)$. Perhaps due to his plan to publicise the main paper subsequently, he veiled all his computations and condensed an enormous message into the following conjecture.

\medskip

\noindent \textbf{Riemann hypothesis (RH):} All zeros of $\Lambda(s)$ lie on the critical line $\Re(s) = \frac{1}{2}$.

\medskip

At first sight, RH is solely an interesting property of $\zeta(s)$, and Riemann himself appears~to take that view. Nonetheless, one should not draw from his thoughts the conclusion that RH is a casual remark of inferior interest. Quintessential examples implying the influence of RH include the result of von Koch~\cite{vonKoch1901}, which states that RH is equivalent to an asymptotic
\begin{equation}\label{von-Koch}
\pi(x) = \mathrm{Li}(x)+O(\sqrt{x} \log x),
\end{equation}
where
\begin{equation*}
\mathrm{Li}(x) \coloneqq \int_{0}^{x} \frac{dt}{\log t}.
\end{equation*}
The estimate~\eqref{von-Koch} is stronger than what the prime number theorem~\eqref{Prime-Number-Theorem} asserts.

Generalisations of RH to broader families of $L$-functions exhibit substantial consequences. We are now in a position to explicate certain extensions of RH leading to the grand Riemann hypothesis. For a primitive Dirichlet character $\chi$ modulo $q$, we define Dirichlet $L$-functions~by
\begin{equation*}
L(s, \chi) \coloneqq \sum_{n = 1}^{\infty} \frac{\chi(n)}{n^{s}} = \prod_{p} \left(1-\frac{\chi(p)}{p^{s}} \right)^{-1}, 
\qquad \Re(s) > 1.
\end{equation*}
It extends to a meromorphic function on $\mathbb{C}$, and obeys the functional equation
\begin{equation*}
\Lambda(s, \chi) \coloneqq \pi^{-\frac{s+\kappa(\chi)}{2}} \Gamma \left(\frac{s+\kappa(\chi)}{2} \right) L(s, \chi)
 = \epsilon(\chi) q^{\frac{1}{2}-s} \Lambda(1-s, \overline{\chi}),
\end{equation*}
where $\kappa(\chi) = \frac{1-\chi(-1)}{2}$ and $\epsilon(\chi)$ denotes the sign of the Gau{\ss} sum $\tau(\chi)$. They give rise to all the degree one $L$-functions. In terms of an ad\`{e}lic language, one can view $n \mapsto \chi(n) n^{-it}$ as a general multiplicative character on the ad\`{e}les; hence the imaginary coordinate $t$ and the Dirichlet character $\chi$ are the archimedean and non-archimedean components respectively of a single ad\`{e}lic frequency parameter. This viewpoint is introduced in Tate's thesis~\cite{Tate1950}.

Families of degree two $L$-functions involve Ramanujan's $\tau$-function given by
\begin{equation*}
\Delta(z) \coloneqq e^{2\pi iz} \prod_{n = 1}^{\infty} (1-e^{2\pi inz})^{24} = \sum_{n = 1}^{\infty} \tau(n) e^{2\pi inz}, 
\qquad z \in \mathbb{H},
\end{equation*}
where $\mathbb{H} \coloneqq \{z = x+iy \in \mathbb{C}: y > 0 \}$ is the upper half-plane upon which the modular~group $\mathrm{SL}_{2}(\mathbb{Z})$ acts via M\"{o}bius transformations. The function $\Delta(z)$ is a holomorphic cusp form of weight $12$ for $\mathrm{SL}_{2}(\mathbb{Z})$. The corresponding $L$-function is defined by
\begin{equation*}
L(s, \Delta) \coloneqq \sum_{n = 1}^{\infty} \frac{\tau(n)}{n^{s}}
 = \prod_{p} \left(1-\frac{\tau(p)}{p^{s}}-\frac{1}{p^{2s-11}} \right)^{-1}, \qquad \Re(s) > \frac{13}{2}.
\end{equation*}
It extends to a meromorphic function on $\mathbb{C}$, and obeys the functional equation
\begin{equation*}
\Lambda(s, \Delta) \coloneqq (2\pi)^{-s} \Gamma(s) L(s, \Delta) = \Lambda(12-s, \Delta).
\end{equation*}
We may normalise this $L$-function so that its critical line is $\Re(s) = \frac{1}{2}$.

To introduce Artin $L$-functions, let $K$ be a global field, and let $\rho: \mathrm{Gal}(\overline{K}/K) \rightarrow \mathrm{GL}_{n}(\mathbb{C})$~be an irreducible $n$-dimensional complex representation of the Galois group, where $\overline{K}$ is a fixed algebraic closure of $K$. As a precursor, Artin~\cite{Artin1924} defined $L$-functions $L(s, \rho)$ as a product over prime ideals of $K$ of local factors that at unramified primes have the form
\begin{equation*}
\det(1-\rho(\mathrm{Frob}_{v}) \mathrm{N}(v)^{-s})^{-1},
\end{equation*}
where $\mathrm{Frob}_{v}$ denotes the Frobenius conjugacy class acting on the invariants under the inertia group at $v$. The completed $L$-function $\Lambda(s, \rho)$ involving gamma factors (canonically attached to $\rho$) from~the archimedean places is known to be meromorphic and to satisfy the functional equation (see~\cite{Artin1930,Brauer1947})
\begin{equation*}
\Lambda(s, \rho) = \epsilon(\rho) c(\rho)^{\frac{1}{2}-s} \Lambda(1-s, \overline{\rho}),
\end{equation*}
where $\overline{\rho}$ is the contragredient representation of $\rho$, $c(\rho)$ is the conductor, and $\epsilon(\rho)$ is a complex number of modulus 1, called the root number. It is conjectured that $\Lambda(s, \rho)$ is entire (Artin holomorphy) unless $\rho$ contains the trivial component $\mathbbm{1}$, in which case it has a pole at $s = 1$. A field $K$ of positive characteristic $p$ is an algebraic function field $K/\mathbb{F}_{q}(t)$~for $q = p^{k}$.~Over such $K$, there exists a bijection between $n$-dimensional $\ell$-adic representations~$\rho$ and automorphic representations $\pi$ of $\mathrm{GL}_{n}(\mathbb{A}_{K})$ via the work of Lafforgue~\cite{Lafforgue2002}:
\begin{equation*}
L(s, \rho) = L(s, \pi).
\end{equation*}

\medskip

\noindent \textbf{Grand Riemann hypothesis (GRH):} All zeros of $\Lambda(s, \rho)$ lie on the critical line $\Re(s) = \frac{1}{2}$.

\medskip

Mertens~\cite{Mertens1874} considered the partial Euler products of $\zeta(s)$ and the Dirichlet $L$-function $L(s, \chi_{-4})$ at $s = 1$, where $\chi_{-4}$ denotes the primitive character modulo $4$. Ramanujan~\cite{Ramanujan1997} extended, beyond the boundary, the work of Mertens~\cite{Mertens1874} to $s > 0$. His result then~states conditionally on RH that
\begin{equation}\label{Ramanujan}
\prod_{p \leq x} \left(1-\frac{1}{p^{s}} \right)^{-1}
 = -\zeta(s) \exp \left(\mathrm{Li}(\vartheta(x)^{1-s})+\frac{2s x^{\frac{1}{2}-s}}{(2s-1) \log x}
 + \frac{S_{s}(x)}{\log x}+O \left(\frac{x^{\frac{1}{2}-s}}{(\log x)^{2}} \right) \right),
\end{equation}
where $\frac{1}{2} < s < 1$ is real, $\vartheta(x) = \sum_{p \leq x} \log p$, and
\begin{equation*}
S_{s}(x) = -s \sum_{\Lambda(\rho) = 0} \frac{x^{\rho-s}}{\rho(\rho-s)}.
\end{equation*}
The asymptotic~\eqref{Ramanujan} is contained in his seminal work on Euler products in the~critical strip. Unfortunately, his paper had not been open to the public until it was published in~\cite{Ramanujan1997}, since his original manuscript was partially published as~\cite{Ramanujan1915} due to the confusion during the World War I. This may be why the study of Euler products in the critical strip was~barely developed in the last century. Prior salient breakthroughs on the subject, however, include the study of the Birch--Swinnerton-Dyer conjecture, which will be detailed in Section~\ref{Birch-Swinnerton-Dyer-conjecture}.

DRH is a conjecture that the normalised partial Euler products converge on the critical~line $\Re(s) = \frac{1}{2}$; see Conjecture~\ref{DRH} for a rigorous formulation. As an application, it is known that DRH for the Riemann zeta function $\zeta(s)$ is equivalent to an asymptotic
\begin{equation*}
\vartheta(x) = x+o(\sqrt{x} \log x),
\end{equation*}
which is due to Akatsuka~\cite{Akatsuka2017}. Note that~\eqref{von-Koch} implies by partial summation that~$\vartheta(x) = x+O(\sqrt{x}(\log x)^{2})$. DRH is proven over function~fields~for arithmetic $L$-functions in~\cite{KimuraKoyamaKurokawa2014} and for Selberg zeta~functions associated to congruence subgroups~$\Gamma_{0}(A) \subset \mathrm{PGL}_{2}(\mathbb{F}_{p}[T])$ with $A \in \mathbb{F}_{p}[T]$ in~\cite{KoyamaSuzuki2014}. Furthermore, the first two authors~\cite{KanekoKoyama2022} studied the convergence of Euler products of Selberg zeta functions in the critical strip, and its relation to the error term in the prime geodesic theorem, a geometric counterpart of the prime number theorem~\eqref{Prime-Number-Theorem}.

\subsection*{Acknowledgements}
Part of this work is inspired by the report ``Deep Riemann hypothesis for $\mathrm{GL}_{n}$'' due to an exchange student at Tokyo Institute of Technology in 2012.

\section{Notation and Terminology}
We use a definition of analytic $L$-functions due to Iwaniec--Kowalski~\cite[Chapter~5]{IwaniecKowalski2004}.~We borrow the notation from~\cite{Devin2020} and work over $\mathbb{Q}$ to keep the subsequent discussions cleaner.
\begin{definition}\label{definition}
Let $L(s, f)$ be a complex-valued function of $s \in \mathbb{C}$. One attaches to $f$ the arithmetic conductor $q(f)$. An analytic $L$-function $L(s, f)$ is defined to satisfy the following.
\begin{enumerate}
\item[(i)] A Dirichlet series factorises as an Euler product of degree $d$ for $\Re(s) > 1$:
\begin{equation*}
L(s, f) \coloneqq \sum_{n = 1}^{\infty} \frac{\lambda_{f}(n)}{n^{s}}
 = \prod_{p} \prod_{j = 1}^{n} \left(1-\frac{\alpha_{j, f}(p)}{p^{s}} \right)^{-1},
\end{equation*}
where $\lambda_{f}(1) = 1$ and $\alpha_{j, f}(p) \in \mathbb{C}$ such that $|\alpha_{j, f}(p)| \leq 1$ for all $1 \leq j \leq n$ and $p \nmid q(f)$. The Dirichlet series and Euler product above are absolutely convergent for~$\Re(s) > 1$;
\item[(ii)] A gamma factor with Langlands parameters $\mu_{\pi}(j) \in \mathbb{C}$, $\Re(\mu_{\pi}(j)) > -1$ is defined by
\begin{equation*}
\gamma(s, f) \coloneqq \pi^{-\frac{ds}{2}} \prod_{j = 1}^{n} \Gamma \left(\frac{s+\mu_{\pi}(j)}{2} \right),
\end{equation*}
the analytic conductor of $f$ is defined by
\begin{equation*}
\mathfrak{q}(f) \coloneqq q(f) \prod_{j = 1}^{n} (|\mu_{\pi}(j)|+e),
\end{equation*}
and the completed $L$-function is defined by
\begin{equation*}
\Lambda(s, f) \coloneqq q(f)^{\frac{s}{2}} \gamma(s, f) L(s, f).
\end{equation*}
It admits an analytic continuation to a meromorphic function of order $1$ with at most poles~at $s = 0$ and $s = 1$. It satisfies the functional equation~$\Lambda(s, f) = \epsilon(f) \Lambda(1-s, \overline{f})$ with $|\epsilon(f)| = 1$, where $\Lambda(s, \overline{f})$ denotes the completed $L$-function associated to $L(s, \overline{f})$;
\item[(iii)] The second moment $L$-function is defined for $\Re(s) > 1$ by
\begin{equation*}
L_{2}(s, f) \coloneqq \prod_{p} \prod_{j = 1}^{n} \left(1-\frac{\alpha_{j, f}(p)^{2}}{p^{s}} \right)^{-1}.
\end{equation*}
It extends to a holomorphic nonvanishing function on $\Re(s) = 1$ except for a zero or pole at possibly one point on $\Re(s) = 1$.
\end{enumerate}
\end{definition}

\begin{examples}\mbox{}
\begin{enumerate}
\item The Riemann zeta function $\zeta(s)$ obeys all the conditions in Definition~\ref{definition}.
\item Dirichlet $L$-functions $L(s, \chi)$ obey all the conditions in Definition~\ref{definition}.
\item All degree $2$ $L$-functions are analytic by~\cite{Deligne1974,DeligneSerre1974}. In particular, it is known~that if $E/\mathbb{Q}$ is an elliptic curve, then $L(s, E)$ is analytic by~\cite{BreuilConradDiamondTaylor2001,TaylorWiles1995,Wiles1995}.
\item If we assume $|\alpha_{j, f}(p)| \leq 1$, then standard $L$-functions attached to Maa{\ss} forms on $\mathrm{GL}_{n}$ are analytic. The condition (ii) is known by~\cite{Cogdell2004,GodementJacquet1972}. Furthermore, the second moment $L$-function is the ratio of the symmetric and exterior square $L$-functions:
\begin{equation*}
L_{2}(s, f) = L(s, \mathrm{sym}^{2} f) L(s, \wedge^{2} f)^{-1} = L(s, f \otimes f) L(s, \wedge^{2} f)^{-2}.
\end{equation*}
By~\cite[Theorems 6.1 \& 7.5]{BumpGinzburg1992} for $L(s, \mathrm{sym}^{2} f)$ and~\cite[Theorems 1--3]{BumpFriedberg1990} and~\cite[Theorems 1--2]{JacquetShalika1990} for $L(s, \wedge^{2} f)$, the condition (iii) is justified.
\end{enumerate}
\end{examples}

We allow some $\alpha_{j, f}(v)$ to equal $0$. We usually expect $|\alpha_{j, f}(p)| = 1$ for all but finitely~many~$p$, but it is not assumed here. The parameters $\alpha_{j, f}(p)$ are determined up to permutations of~the indices. If $|\alpha_{j, f}(p)| = 1$ for all unramified primes and $|\alpha_{j, f}(p)| \leq 1$ otherwise, then $L(s, f)$ is said to satisfy the Ramanujan conjecture, yielding $|\lambda_{f}(\ell)| \leq \tau_{n}(\ell)$, where $\tau_{n}(\ell)$ denotes the $d$-fold divisor function. If $\Re(\mu_{\pi}(j)) \geq 0$, then $L(s, f)$ is said to satisfy the Selberg conjecture or the Ramanujan conjecture at the archimedean place.

\section{Deep Riemann Hypothesis for \texorpdfstring{$\mathrm{GL}_{1}$}{}}
We begin with the formulation of DRH for degree one $L$-functions.

\begin{conjecture}[Deep Riemann hypothesis for Dirichlet $L$-functions]\label{conj:DRH-chi}
If $\chi \ne \mathbbm{1}$, $\Re(s) = \frac{1}{2}$, and $m$ is the order of vanishing of $L(s, \chi)$, then
\begin{equation}\label{Dirichlet-L}
\lim_{x \to \infty} \left((\log x)^{m} \prod_{p \leq x} \left(1-\frac{\chi(p)}{p^{s}} \right)^{-1} \right)
 = \frac{L^{(m)}(s,\chi)}{e^{m\gamma}m!} \times
\begin{cases}
	\sqrt{2} & \text{if $s = \frac{1}{2}$ and $\chi^{2} = \mathbbm{1}$},\\
	1 & \text{otherwise}.
\end{cases}
\end{equation}
\end{conjecture}

This implies the convergence of partial Euler products on the critical line $\Re(s) = \frac{1}{2}$ except at zeros of $L(s, \chi)$. Conrad~\cite{Conrad2005} proved that the following statements are equivalent: the limit on the left-hand side of~\eqref{Dirichlet-L} exists for \textit{some} $s$ on $\Re(s) = \frac{1}{2}$, and it exists for \textit{every} $s$ on $\Re(s) = \frac{1}{2}$. Furthermore, the conjecture~\eqref{Dirichlet-L} is equivalent to the estimate
\begin{equation*}
\vartheta(x, \chi) \coloneqq \sum_{p \leq x} \chi(p) \log p = o(\sqrt{x} \log x).
\end{equation*}
Numerical evidence of DRH can be found in~\cite{KimuraKoyamaKurokawa2014}.

Let $E(x,\chi)$ denote the error term in the prime number theorem in arithmetic progressions:
\begin{equation*}
E(x, \chi) \coloneqq \psi(x, \chi)-
	\begin{cases}
	0 & \text{if $\chi \ne \mathbbm{1}$},\\
	x &\text{if $\chi = \mathbbm{1}$},
	\end{cases}
\end{equation*}
where
\begin{equation*}
\psi(x, \chi) \coloneqq \sum \limits_{n \leq x} \chi(n) \Lambda(n)
\end{equation*}
with the von Mangoldt function given by
\begin{equation*}
\Lambda(n) \coloneqq 
	\begin{cases}
	\log p & \text{if $n = p^{k}$},\\
	0 & \text{otherwise}.
	\end{cases}
\end{equation*}
The more the region of the convergence of the partial Euler product of $L(s, \chi)$ is extended, the~better bound for $E(x, \chi)$ is obtained. This situation is summarised in the following table.

\medskip

\begin{center}
	\begin{tabular}{ll}
		\toprule
		Region & $E(x, \chi)$ \\ \midrule \midrule
		$\Re(s) \geq 1$ & $o(x)$ \\ \midrule
		$\Re(s) > \alpha$ & $O(x^{\alpha} (\log x)^{2})$ \\ \midrule
		$\Re(s) > \frac{1}{2}$ (RH) & $O(\sqrt{x}(\log x)^{2})$ \\ \midrule
		$\Re(s) \geq \frac{1}{2}$ (DRH) & $o(\sqrt{x} \log x)$ \\ \bottomrule
	\end{tabular}
\end{center}

\medskip

The first row displays the classical prime number theorem. The second and third rows are well-known. When the Dirichlet character $\chi$ is principal, the zeta or an $L$-function has a pole at $s = 1$ and the corresponding partial Euler product does not converge in the critical~strip. In this case, Akatsuka~\cite{Akatsuka2017} refined the result~\eqref{Ramanujan} of Ramanujan in the following sense.

\begin{theorem}[{Akatsuka~\cite{Akatsuka2017}}]
If we define 
\begin{equation*}
\zeta_{x}(s) \coloneqq \prod_{p \leq x} \left(1-\frac{1}{p^{s}} \right)^{-1},
\end{equation*}
then the following statements are equivalent.
\begin{enumerate}
\item Let $\psi(x) \coloneqq \psi(x, \mathbbm{1})$ for the trivial character $\mathbbm{1}$. Then
\begin{equation*}
\psi(x) = x+o(\sqrt x \log x).
\end{equation*}
\item There exists $\tau \in \mathbb{R}$ such that
\begin{equation}\label{eq:Akatsuka}
\lim_{x \to \infty} \frac{(\log x)^{m} \zeta_{x}(s)}{\displaystyle{\exp \left(\lim_{\varepsilon \downarrow 0} 
\left(\int_{1+\varepsilon}^{x} \frac{du}{u^{s} \log u}-\log \frac{1}{\varepsilon} \right) \right)}}
\end{equation}
converges to a nonzero value, where $m$ is the order of vanishing of $\zeta(s)$ at $s = \frac{1}{2}+i\tau$. 
\item The above limit exists and is nonzero for every $\tau \in \mathbb{R}$.
\end{enumerate}
If the above conditions are valid, then the Riemann hypothesis holds, and~\eqref{eq:Akatsuka} converges to
\begin{equation*}
(s-1) \frac{\zeta^{(m)}(s)}{e^{(m-1)\gamma} m!} \times
    \begin{cases}
    \sqrt{2} & \text{if $\tau = 0$},\\
    1 & \text{otherwise}.
    \end{cases}
\end{equation*}
\end{theorem}

The statement (2) above is viewed as a version of DRH for $\zeta(s)$, because it is equivalent to the asymptotic formula $E(x, \mathbbm{1}) = o(\sqrt x \log x)$; cf. Conjecture~\ref{conj:DRH-chi} and the ensuing remarks.

\section{Birch--Swinnerton-Dyer Conjecture}\label{Birch-Swinnerton-Dyer-conjecture}
The original Birch--Swinnerton-Dyer conjecture~\cite[Equation (A) on page 79]{BirchSwinnertonDyer1965} pertains to the Euler product asymptotics for an elliptic curve $L$-function at $s = 1$. Let $E/\mathbb{Q}$ be an elliptic curve given by the equation $y^{2} = x^{3}+ax+b$ with the conductor $N$. We also denote $a_{p} \coloneqq p+1-\# E(\mathbb{F}_{p})$ for $p \nmid N$ and $a_{p} \coloneqq p-\# E(\mathbb{F}_{p})$ for $p \mid N$, where $E(\mathbb{F}_{p})$ is the set of nonsingular $\mathbb{F}_{p}$-rational points on a minimal~Weierstrass model for $E$ at $p$. We define
\begin{equation*}
L(s, E) \coloneqq \prod_{p \mid N} \left(1-\frac{a_{p}}{p^{s}} \right)^{-1} \prod_{p \nmid N} \left(1-\frac{a_{p}}{p^{s}}+\frac{1}{p^{2s-1}} \right)^{-1}, \qquad \Re(s) > \frac{3}{2}.
\end{equation*}
Birch--Swinnerton-Dyer~\cite{BirchSwinnertonDyer1965} conjectured that $\# E(\mathbb{F}_{p})$ obeys an asymptotic
\begin{equation}\label{BirchSwinnertonDyer}
\prod_{\substack{p \leq x}} \frac{\# E(\mathbb{F}_{p})}{p} \sim A(\log x)^{r}
\end{equation}
for some $A > 0$ dependent on $E$, where $r$ denotes the rank of $E$, and that
\begin{equation}\label{BirchSwinnertonDyer2}
\mathop{\mathrm{ord}}_{s = 1} L(s,E) = r.
\end{equation}
The conjecture~\eqref{BirchSwinnertonDyer2} is one of the Millennium Problems, but the less famous conjecture~\eqref{BirchSwinnertonDyer} is more important. In fact, Goldfeld~\cite{Goldfeld1982} showed that~\eqref{BirchSwinnertonDyer} implies not only~\eqref{BirchSwinnertonDyer2}, but also the Riemann hypothesis for $L(s,E)$. Nowadays, the conjecture~\eqref{BirchSwinnertonDyer} is seen as a pioneering insight towards DRH since the left-hand side of~\eqref{BirchSwinnertonDyer} matches the Euler product of $L(s, E)$ at $s = 1$. When $r = 0$, it implies the convergence of the Euler product at~$s = 1$ if $L(1, E) \ne 0$.

We denote the corresponding partial Euler product by
\begin{equation*}
\mathrm{Prod}(x, E) \coloneqq \prod_{\substack{p \mid N \\ p \leq x}} \left(1-\frac{a_{p}}{p} \right)^{-1}
\prod_{\substack{p \nmid N \\ p \leq x}} \left(1-\frac{a_{p}}{p}+\frac{1}{p} \right)^{-1}
 = \prod_{p \leq x} \frac{1}{\# E(\mathbb{F}_{p})/p}.
\end{equation*}
\begin{theorem}[Goldfeld~\cite{Goldfeld1982}]\label{Goldfeld}
Let $E/\mathbb{Q}$ be an elliptic curve. If $\mathrm{Prod}(x, E) \sim C(\log x)^{-r}$ for $C > 0$, then $L(s, E)$ satisfies GRH, namely $L(s, E) \ne 0$ for $\Re(s) > 1$, $r = \mathop{\mathrm{ord}}_{s = 1} L(s, E)$, and $C = L^{(r)}(1, E)/(\sqrt{2} e^{r \gamma} r!)$.
\end{theorem}

Conrad~\cite{Conrad2005} studied partial Euler products for analytic $L$-functions on the critical line and demystified the $\sqrt{2}$ factor in terms of the second moment $L$-functions in Definition~\ref{definition}. In the meantime, Kuo--Murty~\cite{KuoMurty2005} proved the equivalence between the Birch--Swinnerton-Dyer conjecture and a certain growth condition for the sum over Frobenius eigenvalues at $p$. We also direct the reader to the work of Sheth~\cite{Sheth2023} for a generalisation of~\eqref{Ramanujan} to~$L(s, E)$.

\section{Deep Riemann Hypothesis for \texorpdfstring{$\mathrm{GL}_{n}$}{}}
This section reviews more general partial Euler products\footnote{The notation $\rho$ for Artin representations should not induce confusion with nontrivial zeros of $\zeta(s)$.}
\begin{equation}\label{partial}
\prod_{\mathrm{N}(v) \leq x} \det(1-\rho(\mathrm{Frob}_{v}) \mathrm{N}(v)^{-s})^{-1}.
\end{equation}
When $\rho \ne \mathbbm{1}$, GRH for Artin $L$-functions would follow if~\eqref{partial} converges for $\frac{1}{2} \leq \Re(s) < 1$. We also address the behaviour on the critical line $\Re(s) = \frac{1}{2}$ for (not necessarily nontrivial) irreducible representations. The property $L(s, \rho_{1} \oplus \rho_{2}) = L(s, \rho_{1}) L(s, \rho_{2})$ allows us to restrict to irreducible representations $\rho$. Kurokawa~\cite{Kurokawa2012,Kurokawa2013} posed the following conjecture.

\begin{conjecture}[Deep Riemann hypothesis for Artin $L$-functions]\label{DRH}
For a global field $K$ and a nontrivial irreducible representation $\rho: \mathrm{Gal}(K^{\mathrm{sep}}/K) \rightarrow \mathrm{GL}_{n}(\mathbb{C})$, we have the following.
\begin{enumerate}
\item \label{condition1}
The limit
\begin{equation*}
\lim_{x \to \infty} \prod_{\mathrm{N}(v) \leq x} \det(1-\rho(\mathrm{Frob}_{v}) \mathrm{N}(v)^{-s})^{-1}
\end{equation*}
exists and converges to $L(s, \rho)$ for $\Re(s) > \frac{1}{2}$.
\item \label{condition2}
If $\Re(s) = \frac{1}{2}$ and $m$ is the order of vanishing of $L(s, \rho)$, then
\begin{equation*}
\lim_{x \to \infty} \left((\log x)^{m} \prod_{\mathrm{N}(v) \leq x} 
\det(1-\rho(\mathrm{Frob}_{v}) \mathrm{N}(v)^{-s})^{-1} \right)
 = \frac{L^{(m)}(s, \rho)}{e^{m \gamma} m!} \times
	\begin{cases}
	\sqrt{2}^{\nu(\rho)} & \text{if $s = \frac{1}{2}$},\\
	1 & \text{otherwise},
	\end{cases}
\end{equation*}
where $\nu(\rho) \coloneqq m(\mathrm{sym}^{2} \rho)-m(\wedge^{2} \rho) \in \mathbb{Z}$ with $m(\rho)$ the multiplicity of $\mathbbm{1}$ in $\rho$.
\end{enumerate}
\end{conjecture}

\begin{theorem}\label{Artin}
Conjecture~\ref{DRH} holds when $\mathrm{char}(K) > 0$.
\end{theorem}

\begin{proof}
We invoke the factorisation due to Weil~\cite{Weil1948}:
\begin{equation*}
L(s, \rho) = \prod_{\lambda} (1-\lambda q^{-s}),
\end{equation*}
where $|\lambda| = \sqrt{q}$. Then we claim the trace formula
\begin{equation}\label{trace-formula}
\sum_{\deg(v) \mid \ell} \deg(v) 
\mathrm{tr}(\rho(\mathrm{Frob}_{v})^{\frac{\ell}{\deg(v)}})
 = -\sum_{\lambda} \lambda^{\ell}.
\end{equation}
This follows from expanding $\log L(s, \rho)$ in two different ways. On the one hand, it equals
\begin{equation*}
\log \left(\prod_{v} \det(1-\rho(\mathrm{Frob}_{v}) 
\mathrm{N}(v)^{-s})^{-1} \right)
 = \sum_{\ell = 1}^{\infty} \frac{1}{\ell} \left(\sum_{\deg(v) \mid \ell} \deg(v) 
\mathrm{tr}(\rho(\mathrm{Frob}_{v})^{\frac{\ell}{\deg(v)}}) \right) q^{-\ell s}.
\end{equation*}
On the other hand, there holds
\begin{equation*}
\log L(s, \rho) = \log \left(\prod_{\lambda} (1-\lambda q^{-s}) \right)
 = \sum_{\ell = 1}^{\infty} \frac{1}{\ell} \left(-\sum_{\lambda} \lambda^{\ell} \right) q^{-\ell s}.
\end{equation*}
The comparison of the two expressions gives~\eqref{trace-formula}. To verify Conjecture~\ref{DRH} (1), we compute
\begin{align*}
&\log \left(\prod_{\mathrm{N}(v) \leq q^{m}} 
\det(1-\rho(\mathrm{Frob}_{v}) \mathrm{N}(v)^{-s})^{-1} \right)\\
& \qquad = \sum_{1 \leq k \deg(v) \leq m} 
\frac{\mathrm{tr}(\rho(\mathrm{Frob}_{v})^{k})}{k} q^{-k \deg(v) s}
 + \sum_{k = 2}^{\infty} \frac{1}{k} \sum_{\frac{m}{k} < \deg(v) \leq m} 
\mathrm{tr}(\rho(\mathrm{Frob}_{v})^{k}) q^{-k \deg(v) s}.
\end{align*}
Note that $\sum_{v} q^{-\deg(v) \alpha} < \infty$ for $\alpha > 1$ and $|\mathrm{tr}(\rho(\mathrm{Frob}_{v})^{k})| \leq n$ since $\rho(\mathrm{Frob}_{v})$ is conjugate~to a unitary matrix. Hence the second term is $o(1)$ as $m \to \infty$. If $\ell = k\deg(v)$, then the~first term becomes
\begin{equation*}
\sum_{\ell = 1}^{m} \frac{1}{\ell} \left(\sum_{\deg(v) \mid \ell} \deg(v) 
\mathrm{tr}(\rho(\mathrm{Frob}_{v})^{\frac{\ell}{\deg(v)}}) \right) q^{-\ell s}
 = \sum_{\ell = 1}^{m} \frac{1}{\ell} \left(-\sum_{\lambda} \lambda^{\ell} \right) q^{-\ell s}.
\end{equation*}
The desired claim now follows from
\begin{equation*}
\sum_{\ell = 1}^{m} \frac{1}{\ell} \left(-\sum_{\lambda} \lambda^{\ell} \right) q^{-\ell s}
 = -\sum_{\lambda} \sum_{\ell = 1}^{\infty} \frac{1}{\ell} \left(\frac{\lambda}{q^{s}} \right)^{\ell}+o(1)
 = \log L(s, \rho)+o(1).
\end{equation*}

To verify Conjecture~\ref{DRH} (2), we write
\begin{equation*}
\log \left(\prod_{\mathrm{N}(v) \leq q^{m}} \det(1-\rho(\mathrm{Frob}_{v}) 
\mathrm{N}(v)^{-\frac{1}{2}})^{-1} \right) = \sum_{\deg(v) \leq m} 
\sum_{k = 1}^{\infty} \frac{\mathrm{tr}(\rho(\mathrm{Frob}_{v})^{k})}{k} q^{-\frac{k}{2} \deg(v)}
 = \text{I}+\text{II}+\text{III},
\end{equation*}
where
\begin{align*}
\text{I}   &\coloneqq \sum_{k \deg(v) \leq m} \frac{\mathrm{tr}(\rho(\mathrm{Frob}_{v})^{k})}{k} q^{-\frac{k}{2} \deg(v)},\\
\text{II}  &\coloneqq \frac12 \sum_{\frac{m}{2} < \deg(v) \leq m} \mathrm{tr}(\rho(\mathrm{Frob}_{v})^{2}) q^{-\deg(v)},\\
\text{III} &\coloneqq \sum_{k = 3}^{\infty} \sum_{\frac{m}{k} < \deg(v) \leq m} 
\frac{\mathrm{tr}(\rho(\mathrm{Frob}_{v})^{k})}{k} q^{-\frac{k}{2} \deg(v)}.
\end{align*}
By the same reason as above, we have $\text{III}=o(1)$ as $m \to \infty$. Furthermore, if $\nu$ denotes the multiplicity of $\lambda = \sqrt{q}$, then the trace formula~\eqref{trace-formula} implies
\begin{equation*}
\text{I} = \sum_{\ell = 1}^{m} \frac{1}{\ell} \left(\sum_{\deg(v) \mid \ell} \deg(v) 
\mathrm{tr}(\rho(\mathrm{Frob}_{v})^{\frac{\ell}{\deg(v)}}) \right) q^{-\frac{\ell}{2}}\\
 = -\nu \sum_{\ell = 1}^{m} \frac{1}{\ell} -\sum_{\lambda \ne \sqrt{q}} 
\sum_{\ell = 1}^{m} \frac{1}{\ell} \left(\frac{\lambda}{\sqrt{q}} \right)^{\ell}.
\end{equation*}
This yields the expression
\begin{equation*}
\text{I} = -\nu \log m-\nu \gamma+\log \left(\prod_{\lambda \ne \sqrt{q}} (1-\lambda q^{-\frac{1}{2}}) \right)+o(1),
\end{equation*}
because for $\lambda \ne \sqrt{q}$,
\begin{equation*}
-\sum_{\ell = 1}^{m} \frac{1}{\ell} \left(\frac{\lambda}{\sqrt{q}} \right)^{\ell} 
 = -\sum_{\ell = 1}^{\infty} \frac{1}{\ell} \left(\frac{\lambda}{\sqrt{q}} \right)^{\ell}+o(1)
 = \log(1-\lambda q^{-\frac{1}{2}})+o(1).
\end{equation*}
To approximate $\text{II}$, the identity $\mathrm{tr}(\rho(\mathrm{Frob}_{v})^{2}) = \mathrm{tr}(\mathrm{sym}^{2} \rho(\mathrm{Frob}_{v}))-\mathrm{tr}(\wedge^{2} \rho(\mathrm{Frob}_{v}))$ implies
\begin{equation*}
\text{II} = \frac{1}{2}(\mathcal{T}(q^{m}, \mathrm{sym}^{2} \rho)-\mathcal{T}(q^{m}, \wedge^{2} \rho)
 - \mathcal{T}(q^{\frac{m}{2}}, \mathrm{sym}^{2} \rho)+\mathcal{T}(q^{\frac{m}{2}}, \wedge^{2} \rho)),
\end{equation*}
where $\mathcal{T}(x, \rho)$ is defined by~\eqref{T}. It now follows from~\cite[Theorem 5]{Rosen1999} or Lemma~\ref{Mertens}~that
\begin{equation*}
\text{II} 
 = \frac{1}{2}(m(\mathrm{sym}^{2} \rho)
 - m(\wedge^{2} \rho)) \log 2+o(1).
\end{equation*}
This completes the proof of Conjecture~\ref{DRH} (2).
\end{proof}

\begin{lemma}\label{Mertens}
For a complex representation $\rho$ of $\mathrm{Gal}(K^{\mathrm{sep}}/K)$, we have that
\begin{equation}\label{T}
\mathcal{T}(x, \rho) \coloneqq \sum_{\mathrm{N}(v) \leq x} \mathrm{tr}(\rho(\mathrm{Frob}_{v})) \mathrm{N}(v)^{-1}
 = m(\rho) \log \log x+C(\rho)+o(1)
\end{equation}
as $m \to \infty$, where $C(\rho)$ is a constant independent of $x$.
\end{lemma}

\begin{proof}
By the factorisation $\rho = \mathbbm{1}^{\oplus m(\rho)} \oplus \rho_{0}$ for $m(\rho_{0}) = 0$, it suffices to establish
\begin{equation}\label{star21}
\mathcal{T}(x, \mathbbm{1}) = \log \log x+C(\mathbbm{1})+o(1)
\end{equation}
and
\begin{equation}\label{star22}
\mathcal{T}(x, \rho_{0}) = C(\rho_{0})+o(1).
\end{equation}
To prove~\eqref{star21}, we employ a version of the prime number theorem (see~\cite[Page~12]{Mitsui1956} for $\mathrm{char}(K) = 0$ and \cite[Theorem~5.12]{Rosen2002} for $\mathrm{char}(K)  > 0$)
\begin{equation*}
\pi_K(x) \coloneqq \sum_{\mathrm{N}(v) \leq x} 1 
 = \mathrm{Li}(x)+O(xe^{-c\sqrt{\log x}})
 = \frac{x}{\log x}+O \left(\frac x{(\log x)^{2}} \right).
\end{equation*}
It follows from partial summation that
\begin{align*}
\sum_{\mathrm{N}(v) \leq x} \frac{1}{\mathrm{N}(v)}
& = \frac{\pi_{K}(x)}{x}-\int_{2}^{x} \left(\frac{t}{\log t}+O \left(\frac{t}{(\log t)^{2}} \right) \right) \left(-\frac{1}{t^{2}} \right) dt\\
& = \frac{1}{\log x}+O \left(\frac{1}{(\log x)^2} \right)+\int_{2}^{x} \left(\frac{1}{t \log t}+O \left(\frac{1}{t(\log t)^2} \right) \right)dt\\
& = \log \log x +C(\mathbbm{1})+o(1)
\end{align*}
for some constant $C(\mathbbm{1})$. The asymptotic~\eqref{star22} is equivalent to the nonvanishing $L(1, \rho) \ne 0$; see~\cite[Corollary 5.47]{IwaniecKowalski2004} for $\mathrm{char}(K) = 0$ and~\cite[Theorem 5]{Rosen1999} for $\mathrm{char}(K) > 0$.
\end{proof}

To a prime $p$ outside the finite set where $\pi_{p}$ is unramified, we assign a semisimple conjugacy class $\mathrm{diag}(\mu_{\pi}(1), \dots, \mu_{\pi}(n))$. Such a class is parametrised by its eigenvalues $\alpha_{j, \pi}(p)$, $1 \leq j \leq n$, called the Satake parameters~\cite{Satake1963} for $\pi_{p}$ unramified (the local factors are represented~by the Langlands parameters of $\pi_{p}$ at the ramified primes). This~gives the definition of the local factors at the unramified primes:
\begin{equation*}
L(s, \pi_{p}) \coloneqq \prod_{j = 1}^{n} (1-\alpha_{j, \pi}(p) p^{-s})^{-1}.
\end{equation*}
They are normalised so that the Ramanujan conjecture states $|\alpha_{j, \pi}(p)| = 1$ and $\Re(\mu_{\pi}(j)) = 0$. As a result, the standard $L$-function associated to $\pi$ boils down to
\begin{equation*}
L(s, \pi) \coloneqq \sum_{n = 1}^{\infty} \frac{\lambda_{\pi}(n)}{n^{s}}
 = \prod_{p} \prod_{j = 1}^{n} (1-\alpha_{j, \pi}(p) p^{-s})^{-1}.
\end{equation*}

In what follows, we work over general number fields to state the most general form of~DRH for standard $L$-functions. In a similar fashion, we define the standard $L$-function over $K$ by
\begin{equation*}
L(s, \pi) \coloneqq \prod_{v} \det(1-M_{\pi}(v) \mathrm{N}(v)^{-s})^{-1},
\end{equation*}
where $v$ runs through prime ideals of $K$ with the Satake parameters $M_{\pi}(v) \in \mathrm{Conj}(\mathrm{GL}_{n}(\mathbb{C}))$.
\begin{conjecture}\label{DRH2}
For a global field $K$ and a nontrivial cuspidal automorphic representation $\pi$ of $\mathrm{GL}_{n}(\mathbb{A}_{K})$, we have the following.
\begin{enumerate}[(1)]
\item The limit
\begin{equation*}
\lim_{x \to \infty} \prod_{\mathrm{N}(v) \leq x} \det(1-M_{\pi}(v) \mathrm{N}(v)^{-s})^{-1}
\end{equation*}
exists and converges to $L(s, \pi)$ for $\Re(s) > \frac{1}{2}$.
\item The limit
\begin{equation*}
\lim_{x \to \infty} \prod_{\mathrm{N}(v) \leq x} \det(1-M_{\pi}(v) \mathrm{N}(v)^{-\frac{1}{2}})^{-1}
\end{equation*}
converges to a nonzero value if and only if $L(\frac{1}{2}, \pi) \ne 0$.
\end{enumerate}
\end{conjecture}

The work of Lafforgue~\cite{Lafforgue2002} implies the following.
\begin{theorem}\label{automorphic}
Conjecture~\ref{DRH2} holds when $\mathrm{char}(K) > 0$.
\end{theorem}

\section{Applications to the Chebyshev Bias}
We delve into implications due to Aoki--Koyama~\cite{AokiKoyama2023} and Koyama--Kurokawa~\cite{KoyamaKurokawa2022}. DRH clarifies that the Chebyshev bias makes a well-balanced disposition of the sequence of the primes in terms of the convergence of Euler products at the centre. Aoki--Koyama~\cite{AokiKoyama2023} addressed the deflection conditionally on DRH -- a new formulation of the Chebyshev bias.

In 1853, Chebyshev noticed in a letter to Fuss that primes congruent to $3$ modulo $4$ seem to dominate over those congruent to $1$ modulo $4$. If $\pi(x; q, a)$ is the number of primes $p \leq x$ such that $p \equiv a \tpmod q$, then the inequality $\pi(x; 4, 3) \geq \pi(x; 4, 1)$ holds for more than 97~\% of $x < 10^{11}$. By a classical theorem of Dirichlet, we expect that the number of primes of the form $4k+1$ and $4k+3$ ought to be asymptotically equal. Hence, the Chebyshev~bias~shows that primes of the form $4k+3$ appear earlier than those of the form $4k+1$. Littlewood~\cite{Littlewood1914} established the infinitude of sign changes of $\pi(x; 4, 3)-\pi(x; 4, 1)$. Furthermore, Knapowski--Tur\'{a}n~\cite{KnapowskiTuran1962} predicted that the density of numbers $x$ for which $\pi(x; 4, 3) \geq \pi(x; 4, 1)$ holds is $1$, while Kaczorowski~\cite{Kaczorowski1995} disproved their conjecture conditionally on GRH. They have a logarithmic density due to Rubinstein--Sarnak~\cite{RubinsteinSarnak1994}, which is approximately $0.9959 \cdots$.

The work of Aoki--Koyama~\cite{AokiKoyama2023} introduces the weighted counting function
\begin{equation*}
\pi_{s}(x; q, a) \coloneqq \sum_{\substack{p < x \\ p \equiv a \tpmod q}} \frac{1}{p^{s}}, \qquad s \geq 0,
\end{equation*}
where the smaller prime $p$ allows a higher contribution to $\pi_s(x; q, a)$. The function $\pi_{s}(x; q, a)$ is more proper than $\pi(x; q, a)$ to represent the phenomenon since it reflects the size of primes that $\pi(x; q, a)$ ignores. While the natural density of the set
\begin{equation*}
A(s) \coloneqq \{x > 0: \pi_{s}(x; 4, 3)-\pi_{s}(x; 4, 1) \geq 0 \}
\end{equation*}
does not exist when $s=0$, they proved under DRH that it is equal to 1 when $s = \frac{1}{2}$, namely
\begin{equation*}
\lim_{X \to \infty} \frac{1}{X} \int_{t \in A(\frac{1}{2}) \cap [2,X]} dt = 1.
\end{equation*}
More precisely, the existence of a constant $C$ such that
\begin{equation}\label{CB}
\pi_{\frac{1}{2}}(x; 4, 3)-\pi_{\frac{1}{2}}(x; 4, 1) = \frac{1}{2} \log \log x+C+o(1)
\end{equation}
is equivalent to a weak form of DRH for the Dirichlet $L$-function $L(s,\chi_{-4})$.

\begin{conjecture}[Convergence conjecture]\label{CC}
The Euler product of $L(s,\rho)$ converges at $s = \frac{1}{2}$, namely there exists a constant $L$ such that
\begin{equation*}
(\log x)^{m} \prod_{\mathrm{N}(v) \leq x} \det(1-\rho(\mathrm{Frob}_{v}) \mathrm{N}(v)^{-\frac{1}{2}})^{-1}
 = L+o(1)
\end{equation*}
as $x \to \infty$, where $m = \mathrm{ord}_{s = \frac{1}{2}} L(s,\rho)$.
\end{conjecture}

By~\cite[Theorem 5.3]{Conrad2005}, Conjecture~\ref{CC} is stronger than GRH implying the convergence of Euler products for $\Re(s) > \frac{1}{2}$. The proof of the equivalence between~\eqref{CB} and Conjecture~\ref{CC} looks like the following. Because $m = 0$ when $\rho = \chi_{-4}$, Conjecture~\ref{CC} is equivalent to
\begin{equation*}
\prod_{p \leq x} \left(1-\frac{\chi_{-4}(p)}{\sqrt{p}} \right)^{-1} = L+o(1),
\end{equation*}
where $L \ne 0$. This asymptotic is recast as
\begin{equation}\label{conv}
\sum_{p \leq x} \log \left(1-\frac{\chi_{-4}(p)}{\sqrt{p}} \right)^{-1} = \log L+o(1).
\end{equation}
Upon expanding the left-hand side as
\begin{equation*}
\sum_{k = 1}^{\infty} \sum_{p \leq x} \frac{\chi_{-4}(p)^{k}}{kp^{\frac{k}{2}}},
\end{equation*}
the subseries over $k \geq 3$ is absolutely convergent as $x \to \infty$. On the other hand, the subseries over~$k = 2$ satisfies by Euler's theorem
\begin{equation}\label{Euler}
\sum_{p \leq x} \frac{\chi_{-4}(p)^{2}}{2p}
 = \sum_{p \leq x} \frac{1}{2p}
 = \frac{1}{2} \log \log x+c+o(1)
\end{equation}
for some $c \in \mathbb{R}$. By~\eqref{conv} and~\eqref{Euler}, the behavior corresponding to $k = 1$ becomes
\begin{equation*}
\sum_{p \leq x} \frac{\chi_{-4}(p)}{\sqrt{p}}
 = -\frac{1}{2} \log \log x+\log L-c+o(1).
\end{equation*}

We now formulate the Chebyshev biases for prime ideals $\mathfrak{p}$ of a global field $K$.
\begin{definition}[Aoki--Koyama~\cite{AokiKoyama2023}]\label{definition-1}
Let $c(\mathfrak{p}) \in \mathbb{R}$ traverse a sequence over prime ideals $\mathfrak{p}$ of $K$ such~that
\begin{equation*}
\lim_{x \to \infty} \frac{\#\{\mathfrak{p} \mid c(\mathfrak{p}) > 0, \mathrm{N}(\mathfrak{p}) \leq x \}}
{\#\{\mathfrak{p} \mid c(\mathfrak{p}) < 0, \mathrm{N}(\mathfrak{p}) \leq x \}} = 1.
\end{equation*}
We say that $c(\mathfrak{p})$ has a Chebyshev bias towards being positive if there exists a constant $C > 0$ such that
\begin{equation*}
\sum_{\mathrm{N}(\mathfrak{p}) \leq x} \frac{c(\mathfrak{p})}{\sqrt{\mathrm{N}(\mathfrak{p})}} \sim C \log \log x,
\end{equation*}
where $\mathfrak{p}$ runs through prime ideals of $K$. On the other hand, we say that $c(\mathfrak{p})$ is unbiased if
\begin{equation*}
\sum_{\mathrm{N}(\mathfrak{p}) \leq x} \frac{c(\mathfrak{p})}{\sqrt{\mathrm{N}(\mathfrak{p})}} = O(1).
\end{equation*}
\end{definition}

\begin{definition}[Aoki--Koyama~\cite{AokiKoyama2023}]\label{definition-2}
Assume that the set of all prime ideals $\mathfrak{p}$ of $K$ with $\mathrm{N}(\mathfrak{p}) \leq x$ is a disjoint union $P_{1}(x) \cup P_{2}(x)$ and that their proportion converges to
\begin{equation*}
\delta = \lim_{x \to \infty} \frac{|P_{1}(x)|}{|P_{2}(x)|}.
\end{equation*}
We say that there exists a Chebyshev bias towards $P_{1}$ or a Chebyshev bias against $P_{2}$ if there exists a constant $C > 0$ such that
\begin{equation*}
\sum_{\mathfrak{p} \in P_{1}(x)} \frac{1}{\sqrt{\mathrm{N}(\mathfrak{p})}}
 - \delta \sum_{\mathfrak{p} \in P_{2}(x)} \frac{1}{\sqrt{\mathrm{N}(\mathfrak{p})}}
\sim C \log \log x
\end{equation*}
for some $C > 0$. On the other hand, we say that there exist no biases between $P_{1}$ and $P_{2}$ if
\begin{equation*}
\sum_{\mathfrak{p} \in P_{1}(x)} \frac{1}{\sqrt{\mathrm{N}(\mathfrak{p})}}
 - \delta \sum_{\mathfrak{p} \in P_{2}(x)} \frac{1}{\sqrt{\mathrm{N}(\mathfrak{p})}} = O(1).
\end{equation*}
\end{definition}

Let $L/K$ be a finite Galois extension of global fields. The results of Aoki--Koyama~\cite{AokiKoyama2023} yield versions of Chebyshev biases existing in the prime ideals of $K$. We here consider some salient examples. Let $S$ be the set of all prime ideals in $K$, and let $S_{\sigma} \subset S$ be the subset of unramified primes whose Frobenius element $(\frac{L/K}{\mathfrak{p}})$ is equal to $\sigma \in \mathrm{Gal}(L/K)$.

\begin{example}[{Aoki--Koyama~\cite[Theorem 2.2]{AokiKoyama2023}}]
Let $L/K$ be a finite Galois extension~of global fields, and let $\rho$ traverse a nontrivial irreducible representation of $\mathrm{Gal}(L/K)$. Then~the following statements are equivalent.
\begin{enumerate}[\textrm(i)]
\item[(i)] Conjecture~\ref{CC} holds for $L(s, \rho)$.
\item[(ii)] For any $\sigma \in \mathrm{Gal}(L/K)$, there exist constants $C$ and $c$ dependent on $\sigma$ such that
\begin{equation*}
\sum_{\substack{\mathfrak{p} \in S \\ \mathrm{N}(\mathfrak{p}) \leq x}} \frac{1}{\sqrt{\mathrm{N}(\mathfrak p)}}
-\frac{[L:K]}{|c_{\sigma}|} \sum_{\substack{\mathfrak{p} \in S_{\sigma} \\ \mathrm{N}(\mathfrak{p}) \leq x}} 
\frac{1}{\sqrt{\mathrm{N}(\mathfrak p)}} = C \log \log x+c+o(1),
\end{equation*}
where $c_{\sigma}$ denotes the conjugacy class~of~$\sigma$.
\end{enumerate}
\end{example}

\begin{example}[{Aoki--Koyama~\cite[Example 3.3]{AokiKoyama2023}}]\label{ex:splitting}
Assume that $[L:K] = 2$, and let $\chi$ be the nontrivial character of $\mathrm{Gal}(L/K)$. Then the following statements are equivalent.
\begin{enumerate}[\textrm(i)]
\item Conjecture~\ref{CC} holds for $L(s,\chi)$.
\item There exists a Chebyshev bias against splitting primes, namely there exists a constant $c$ such that
\begin{equation*}
\sum_{\substack{\mathfrak{p} \colon \text{nonsplit} \\ \mathrm{N}(\mathfrak{p}) \leq x}} \frac{1}{\sqrt{\mathrm{N}(\mathfrak{p})}}
 - \sum_{\substack{\mathfrak{p} \colon \text{split} \\ \mathrm{N}(\mathfrak{p}) \leq x}} \frac{1}{\sqrt{\mathrm{N}(\mathfrak{p})}}
 = \left(\frac{1}{2}+m(\chi) \right) \log \log x+c+o(1).
\end{equation*}
\end{enumerate}
\end{example}

\begin{example}[{Aoki--Koyama~\cite[Corollary 3.2]{AokiKoyama2023}}]\label{ex:residues}
Let $q \in \mathbb{N}$. Assume that $L(\frac{1}{2}, \chi) \ne 0$ for any Dirichlet character $\chi$ modulo $q$. Then the following statements are equivalent.
\begin{enumerate}[\textrm(i)]
\item Conjecture \ref{CC} holds for $L(s,\chi)$.
\item There exists a Chebyshev bias against quadratic residues modulo $q$, namely there exists a constant $c$ such that
\begin{equation*}
\pi_{\frac{1}{2}}(x; q, b)-\pi_{\frac{1}{2}}(x; q, a) = \frac{2^{t-1}}{\varphi(q)} \log \log x+c+o(1)
\end{equation*}
for any pair $(a, b)$ of a quadratic residue $a$ and a quadratic nonresidue $b$, and there exist no biases for all other pairs $(a, b)$.
\end{enumerate}
\end{example}

\begin{example}[{Aoki--Koyama~\cite[Corollary 3.5]{AokiKoyama2023}}]\label{ex:principal}
Let $\widetilde{K}$ be the Hilbert class field of~$K$ whose ideal class group is $\mathrm{Cl}_{K} \simeq \mathrm{Gal}(\widetilde{K}/K)$. An ideal class $[\mathfrak a] \in \mathrm{Cl}_K$ corresponds to~$(\frac{\widetilde{K}/K}{\mathfrak{a}}) \in \mathrm{Gal}(\widetilde{K}/K)$. Assume Conjecture~\ref{CC} for $L(s, \chi)$ and that $L(\frac{1}{2}, \chi) \ne 0$ for any character $\chi$ of $\mathrm{Cl}_{K}$. When $|\mathrm{Cl}_K|$ is even, there exists a Chebyshev bias against principal ideals in the whole sequence of prime ideals of $K$, namely there exists a constant $c$ such that
\begin{equation*}
\sum_{\substack{\mathfrak{p} \colon \text{nonprincipal} \\ \mathrm{N}(\mathfrak{p}) \leq x}} 
\frac{1}{\sqrt{\mathrm{N}(\mathfrak{p})}}-(|\mathrm{Cl}_{K}|-1) \sum_{\substack{\mathfrak{p} \colon \text{principal} \\ 
\mathrm{N}(\mathfrak{p}) \leq x}} \frac{1}{\sqrt{\mathrm{N}(\mathfrak{p})}}
=\frac{|\mathrm{Cl}_{K}/\mathrm{Cl}_{K}^{2}|-1}{2} \log \log x+c+o(1).
\end{equation*}
\end{example}

Koyama--Kurokawa~\cite{KoyamaKurokawa2022} exemplified a Chebyshev bias for Ramanujan's $\tau$-function by adapting such a phenomenon to the context of degree $2$ $L$-functions.
\begin{example}[{Koyama--Kurokawa \cite[Theorem~2]{KoyamaKurokawa2022}}]\label{tau}
Assume Conjecture~\ref{CC} for $L(s+\frac{11}{2}, \Delta)$. Then the sequence $\tau(p) p^{-\frac{11}{2}}$ has a Chebyshev bias towards being positive, namely there exists a constant $c$ such that
\begin{equation*}
\sum_{p \leq x} \frac{\tau(p)}{p^{6}} = \frac{1}{2} \log \log x+c+o(1).
\end{equation*}
\end{example}

Example~\ref{tau} implies (under the convergence conjecture) that the Satake parameters $\theta(p) \in [0, \pi]$ has a Chebyshev bias~towards being in $[0, \frac{\pi}{2}]$, where $\tau(p) = 2p^{\frac{11}{2}} \cos(\theta(p))$.

\begin{remark}
Sarnak~\cite{Sarnak2007-2} explored a similar phenomenon under the Generalised Riemann hypothesis and the linear independence over $\mathbb{Q}$ of the imaginary parts of all nontrivial zeros of $L(s, \Delta)$ in the upper half-plane. He points out that the sum
\begin{equation*}
S(x) \coloneqq \sum_{p \leq x} \frac{\tau(p)}{p^{\frac{11}{2}}}
\end{equation*}
has a bias towards being positive in the sense that the mean of the measure $\mu$ defined by
\begin{equation*}
\frac{1}{\log X} \int_{2}^{X} f \left(\frac{\log x}{\sqrt{x}} S(x) \right) \frac{dx}{x} \to \int_{\mathbb{R}} f(x) d\mu(x)
\end{equation*}
for $f \in C_{c}^{\infty}(\mathbb{R})$ equals $1$. He studied the logarithmic derivative of $L(s, \Delta)$ and observed~that the second term in its expansion causes the bias. While this discussion handles the logarithm instead, the argument of Aoki--Koyama~\cite{AokiKoyama2023} is more straightforward thanks to DRH.
\end{remark}

\begin{remark}
We refer to the work of the first two authors~\cite{KanekoKoyama2023} and Okumura~\cite{Okumura2023} for studies of the Chebyshev biases for elliptic curves over function fields and Fermat curves of prime degree, respectively, from the perspective of DRH.
\end{remark}


\providecommand{\bysame}{\leavevmode\hbox to3em{\hrulefill}\thinspace}
\providecommand{\MR}{\relax\ifhmode\unskip\space\fi MR }
\providecommand{\MRhref}[2]{%
  \href{http://www.ams.org/mathscinet-getitem?mr=#1}{#2}
}
\providecommand{\href}[2]{#2}

\end{document}